\documentclass[reqno]{amsart}
\usepackage{amsthm, amscd, amsfonts,graphicx}
\usepackage{color}
\usepackage{enumerate}
\usepackage{verbatim}
\numberwithin{equation}{section}
\theoremstyle{plain}
 \newtheorem{theorem}{Theorem}[section]
 \newtheorem{lemma}[theorem]{Lemma}
 \newtheorem{proposition}[theorem]{Proposition}
 
 \newtheorem{corollary}[theorem]{Corollary}
\theoremstyle{definition}
 \newtheorem{definition}[theorem]{Definition}
 
 \newtheorem{remark}[theorem]{Remark}

\theoremstyle{remark}
 \newtheorem{case}{Case}

\newenvironment{enumeratei}{\begin{enumerate}[\quad\upshape (i)]} {\end{enumerate}}

\newcommand \rtau {\mathrel\tau}

\newcommand \nrtau {\mathrel{\overline\tau}}
\newcommand \rpi {\mathrel\pi}
\newcommand \rrho {\mathrel\rho}
\newcommand \pT {\prec_T}
\newcommand \npT {\not\prec_T}
\newcommand \cht {\mathsf C_2} 
\newcommand \dbl[2] {{[#1\ast_{2}#2]}}
\newcommand \defiff {\overset {\textup{def}}\iff}
\newcommand \prho {\prec_\rho}
\newcommand \lrho {<_\rho}
\newcommand \lqrho {\leq_\rho}
\newcommand \ppi {\prec_\pi}
\newcommand \aL {L\times\cht}
\newcommand \cref[1] {\hbox{\textup{(cov-\ref{#1})}}}
\newcommand \fcref[1] {\textup{(cov-$#1$)}}
\newcommand \pref[1] {\textup{\hbox{(cov$^\ast$-\ref{#1})}}}
\newcommand \fpref[1] {\textup{(cov$^\ast$-$#1$)}}
\newcommand \poset[1]{(#1;<)}
\newcommand \alg[1]{{\mathcal #1}}
\newcommand \Pow [1] {\textup{Pow}(#1)}
\newcommand \PowP {\Pow{\alg P}}
\newcommand \CoalLattxt{CoalL}
\newcommand \Coal [1] {\textup{\CoalLattxt}(#1)}
\newcommand \CoalP{\textup{\CoalLattxt}(\alg P)}
\newcommand \pCoalP{\textup{\CoalLattxt}(\alg P')}
\newcommand \whCoalP{\textup{\CoalLattxt}(\widehat{\alg P})}
\newcommand \iCoalP [1]{\textup{\CoalLattxt}(\alg P_{#1})}
\newcommand \ieCoal [1] {\textup{\CoalLattxt}(#1)=(\Pow{#1};\leq)}
\newcommand \ieCoalP {\ieCoal{\alg P}}
\renewcommand \phi{\varphi}
\newcommand \ideal {\mathord{\downarrow}}
\newcommand \filter {\mathord{\uparrow}}
\newcommand\set [1]{\{#1\}}
\newcommand \tbf[1] {\textbf{#1}}

\newcommand \nonparallel {\mathrel{\not\kern-1.5pt{\mathop\parallel}}}

\newcommand \height {h}
\newcommand \idmap[1] {\textup{id}_{#1}}
\newcommand \restrict [2] {#1\rceil_{#2}}
\newcommand \tstr {\textup{str}}
\newcommand \str [1] {\tstr(#1)}
\newcommand \pstr [2] {\tstr_{#1}(#2)}

\newcommand \nlat[1]{(#1;\leq,\vee,\wedge)}

\renewcommand \epsilon {\varepsilon}
\newcommand \qn {qn}

\newcommand \red [1] {{\color{red}#1\color{black}}}

\newcommand \nothing [1] {}
\newcommand \magenta [1] {{\color{magenta}#1\color{black}}}

\begin{document}
\title
[Doubling tolerances and coalition lattices]
{Doubling tolerances and coalition lattices}

\author[G.\ Cz\'edli]{G\'abor Cz\'edli}
\address{University of Szeged, Bolyai Institute, Szeged,
Aradi v\'ertan\'uk tere 1, Hungary 6720}
\email{czedli@math.u-szeged.hu}
\urladdr{http://www.math.u-szeged.hu/~czedli/}

\thanks{This research was supported by NFSR of Hungary (OTKA), grant
number K 115518}

\subjclass{006B99, 06C99, 06D99}




\keywords{Lattice tolerance, modular lattice, coalition lattice}

\dedicatory{Dedicated to the memory of Ivo G. Rosenberg}

\date{\magenta{\emph{\tbf{Always}} check the author's homepage for updates!} \kern 1cm \hfill \red{December 10, 2019}}

\begin{abstract} If every block of a (compatible) tolerance (relation) $T$ on a modular lattice $L$ of finite length consists of at most two elements, then we call $T$ a \emph{doubling tolerance} on $L$. We prove that, in this case, $L$ and $T$ determine a modular lattice of size $2|L|$. This construction preserves distributivity and modularity. In order to give an application of the new construct,  let $P$ be a partially ordered set (poset). Following a 1995 paper by G.\ Poll\'ak and the present author, the subsets of $P$ are called the \emph{coalitions} of $P$. For coalitions $X$ and $Y$ of $P$, let $X\leq Y$ mean that there exists an injective map $f$ from $X$ to $Y$ such that $x\leq f(x)$ for every $x\in X$. If $P$ is a finite chain, then its coalitions form a distributive lattice by the 1995 paper; we give a new proof of the distributivity of this lattice by means of doubling tolerances. 
\end{abstract}

\maketitle

\section{Introduction}
There are two words in the title that are in connection with Ivo G. Rosenberg; these words are ``tolerance'' and ``lattice'', both occurring also in the title  of our joint lattice theoretical paper \cite{chczgigr} (coauthored also by I.\ Chajda). This fact encouraged me to submit the present paper to a special volume dedicated to Ivo's memory even if this volume does not focus on lattice theory.

The paper is structured as follows. In Section \ref{sectiontol}, after few historical comments on lattice tolerances, we introduce the concept of doubling tolerances (on lattices) as those tolerances whose blocks are at most two-element.  We prove that finite modular lattices can be ``doubled'' with the help of this tolerances; see Theorem~\ref{thmdoubling}. Furthermore, this doubling construction preserves modularity and distributivity.
In Section~\ref{sectionfincoal}, after recalling the concept of  coalition lattices of certain finite posets (partially ordered sets) and some related results and after presenting some new observations, we use our doubling construction to give a new proof of the fact that coalition lattices of finite chains are distributive; see Lemma~\ref{lemmazBfmWvVjGtRx}.
Also, this new proof provides a natural example of doubling tolerances and our construct. 

Note in advance that every structure in this paper is assumed to be of \emph{finite length} even if this is not always emphasized.

\section{Doubling tolerances}\label{sectiontol}
Tolerances (that is, compatible tolerance relations) of lattices were first investigated by Chajda and Zelinka~\cite{chajdazelinka}. They are reflexive and symmetric  relations preserved by both lattice operations. Tolerances on lattices have been studied for long; see, for example Bandelt~\cite{bandelt}, Chajda~\cite{chajdabook}, Cz\'edli~\cite{czglperrho}, Cz\'edli and Gr\"atzer~\cite{czggg}, Grygiel and Radeleczki~\cite{grygrad}, and Kindermann~\cite{kindermann}.
Apparently, apart from some artificial constructs like those in  Chajda, Cz\'edli, and  Hala\v s~\cite{chczgrh}, tolerances seem to be interesting only in lattices and lattice-like structures.

A lattice $L$ is of \emph{finite length} if there is a natural number $n$ such that no chain in $L$ has more than $n+1$ elements; if so then the least such $n$ is the \emph{length} of $L$. 
For a tolerance $T$ on a lattice $L$ of finite length, maximal subsets $X$ of $L$ such that $X\times X\subseteq T$ are called the \emph{blocks} of $T$; they are known to be intervals; see, for example, Cz\'edli~\cite{czglperrho}. By \cite[Proposition 1]{czglperrho} and its dual, 
\begin{equation}
\text{for arbitrary blocks $[a,b]$ and $[c,d]$ of $T$, $a=c\iff b=d$}.
\label{eqtxtblTBlnl}
\end{equation}
If $L$ (of finite length as always) is modular and  no block of $T$ has more than two elements or, equivalently, if the ``covering or equal'' relation  $a\preceq b$  holds for every block $[a,b]$ of $T$, then $T$ will be called a \emph{doubling tolerance} on $L$. 
Let us emphasize that we define doubling tolerances only on \emph{modular} lattices \emph{of finite length}t.
If $T$ is a doubling tolerance on $L$, then its intersection with the covering relation  $\prec$ (or, equivalently, with the strict lattice ordering $<$) will be denoted by $\pT$. That is, 
$a\pT b\iff [a,b]$ is a two-element block of $T$. Note that $\pT$ determines $T$.
We will also need the negated relation: $a \npT b$ will mean that $a\pT b$ fails.
The two-element chain will be denoted by $\cht:=\set{0,1}$; with $0\prec 1$, of course. 
The direct product order on $L\times \cht$ will be denoted by $\leq$ or $\pi$. That is, each of $(a,i)\leq (b,j)$,  $(a,i)\rpi (b,j)$, and $((a,i),(b,j))\in\pi$ means that $a\leq b$ in $L$ and $i\leq j$ in $\cht$. On $L\times \cht$, we define a relation $\tau$ as follows: 
\[
\text{$(a,i)\rtau (b,j)$\quad $\defiff$ \quad  $i=1$, $j=0$ and $a\pT b$.}
\] 
We will also need the negated relation:
\[
\text{$(a,i)\nrtau (b,j)$\quad $\defiff$ \quad $(a,i)\rtau (b,j)$ does not hold.}
\]
Let $\rho$ the transitive closure of $\pi\cup\tau$; it is a relation on $\aL$. Finally, we denote the structure $(\aL;\rho)$ by $\dbl L T$. The subscript of $\ast$ reminds us that $|L\times \cht|=2\cdot|L|$, so what we do is \emph{doubling} in some sense. With the concepts and notations introduced above, we are now in the position to formulate our main result. 

\begin{figure}[htb] 
\centerline
{\includegraphics[scale=1.0]{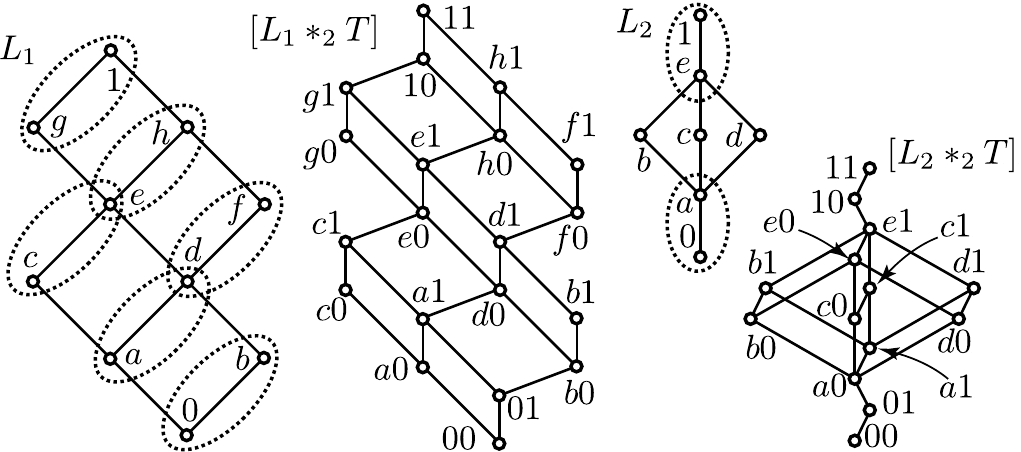}}
\caption{Two examples; the $T$-blocks are the dotted ovals
\label{figura}}
\end{figure}%

\begin{theorem}\label{thmdoubling}
If $T$ is a doubling tolerance on a modular lattice $L$ of finite length, then $\dbl LT=(\aL;\rho)$ is also a modular lattice of finite length. If, in addition, $L$ is distributive, then so is $\dbl LT$.
\end{theorem}

\begin{remark} If $T$ is the equality relation, all of whose blocks are singletons, then $\dbl LT$ is simply the direct product of $L$ and $\cht$. Some examples where $T$ is not the equality relation are given in Figures~\ref{figura} and Figure~\ref{figurb}.
\end{remark}

\begin{figure}[htb] 
\centerline
{\includegraphics[scale=1.0]{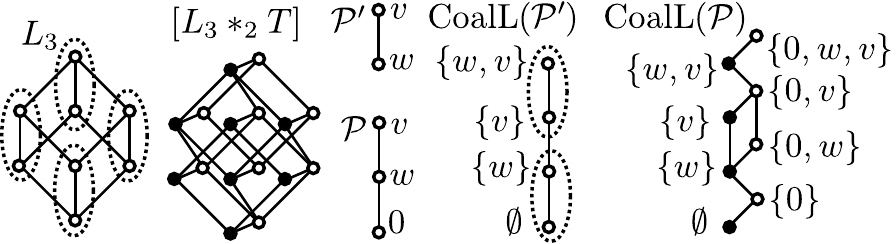}}
\caption{Two further examples; the lower elements are black filled and $\dbl{\pCoalP}T\cong\CoalP$
\label{figurb}}
\end{figure}%

\begin{proof}[Proof of Theorem~\ref{thmdoubling}] Let $T$ be a doubling tolerance on a modular lattice $L$ of finite length. The elements of $L\times\set0$ will be called \emph{lower elements} while those of $L\times \set 1$ are the \emph{upper elements}. 
Observe that, for any $(a,i),(b,j)\in L\times \cht$, 
\begin{equation}
\text{if  $(a,i)\rrho (b,j)$, then $a\leq b$ in $L$;}
\label{eqbzKnknSrVd}
\end{equation}
simply because both $\pi$ and $\tau$ have the same property. 
By a \emph{$(\pi\cup\tau)$-sequence} from $(a,i)\in \aL$ to $(b,j)\in L\times \cht$ we mean a finite sequence
\begin{equation}
 (a,i)=(x_0,k_0) \mathrel{\pi\!\cup\!\tau} (x_1,k_1)
   \mathrel{\pi\!\cup\!\tau} \dots  \mathrel{\pi\!\cup\!\tau} (x_n,k_n)=(b,j)
\label{eqprsrzT}
\end{equation}
of elements of $\aL$. By the definition of $\rho$, we have that $(a,i)\rrho (b,j)$ if and only if there is a sequence described in \eqref{eqprsrzT}.
In order to show that $\rho$ is antisymmetric, assume that $(a,i)\rrho (b,j)$ and $(b,j)\rrho (a,i)$. By \eqref{eqbzKnknSrVd}, $a=b$, whereby the only way of violating antisymmetry is  that $(a,1)\rrho (a,0)$. Take a $(\pi\cup\tau)$-sequence from $(a,1)$ to $(a,0)$; see \eqref{eqprsrzT}. Since only a ``$\tau$-step'' in this sequence can change an upper element to a lower element, at least one $\tau$-step occurs in this sequence. But this step strictly increases the first component, whereby it follows from \eqref{eqbzKnknSrVd} that the first component cannot remain $a$ at the end of the sequence. Thus, $(a,1)\rrho (a,0)$ is impossible and $\rho$ is antisymmetric. The reflexivity of $\rho$ follows from $\rho\subseteq \pi$ while $\rho$ is transitive by its definition. Hence, $\dbl L T$ is a poset; note that the modularity of $L$ has not yet been used. 

Now that we know that $\rho$ is a partial ordering, we can speak about the corresponding covering relation, which will be denoted by $\prho$. Self-explanatory analogous notations, like $\lrho$, $\lqrho$, or $\ppi$ will also be used. If the superscript is dropped, then the meaning should be clear for the context since $\leq$, $\prec$, $<$ refer to the original ordering of $L$ when they are applied for the elements of $L$, and they refer to $\rho$ between elements of $\dbl LT$. 
We claim that, for arbitrary $(a,i),(b,j)\in \aL$, we have that $(a,i)\prho (b,j)$ if and only if one of the following three possibilities holds.
\begin{enumerate}[\qquad({cov-}1)]
\item\label{caa} $a=b$, $i=0$, and $j=1$.
\item\label{cab} $i=j$, $a\prec b$ in $L$, and $a\npT b$  ( that is, $[a,b]$ is not a $T$-block).
\item\label{cta} $i=1$, $j=0$, and $a\pT b$.
\end{enumerate}
In order to verify that the disjunction of \cref{caa}, \cref{cab}, and 
\cref{cta} describes $\prho$ correctly, assume that $(a,i)\prho (b,j)$. Take a \emph{repetition-free} $(\pi\cup\tau)$-sequence from $(a,i)\in \aL$ to $(b,j)$; see \eqref{eqprsrzT}. We can assume that every $\pi$-step in this sequence is a $\prec_\pi$-step. 
Clearly, the sequence consists of a single step, which is either a $\tau$-step, corresponding to \cref{cta}, or it is a $\prec_\pi$-step $(a,0)\prec_\pi (a,1)$ or $(a,i)\prec_\pi (b,i)$ with $a\prec b$ in $L$; furthermore, if $(a,i)\prec_\pi (b,i)$ with $a\prec b$ in $L$, then $a\npT b$ since otherwise $(a,1)\rtau (b,0)$ and $(a,0)<_\rho (a,1)<_\rho (b,0)$ (for $i=j=0$) or 
$(a,1)<_\rho (b,0)<_\rho (b,1)$ (for $i=j=1)$ would contradict $(a,i)\prho (b,j)$. So whenever $(a,i)\prho (b,j)$, then at least one of (and exactly one of) \cref{caa}, \cref{cab}, and \cref{cta} holds. 
Since a $\tau$-step of a $(\pi\cup\tau)$-sequence increases the first component, it follows that \cref{caa}, which a special sort of a $\pi$-covering, is indeed a $\rho$-covering. 
Since $a\npT b$ in \cref{cab}, it follows from \eqref{eqbzKnknSrVd} that no $\tau$-covering step can interfere and the $\pi$-covering described in \cref{cab} is a $\rho$-covering, as required. 
Finally, since  $\rho$ is antisymmetric, we have that for every $x\in L$,
\begin{equation}
(x,1)\not\leq_\rho (x,0).
\label{eqmkZsprmsJlKrNqg}
\end{equation}
Now if  \cref{cta} holds for $(a,i)$ and $(b,j)$, then $a\pT b$ yields that  $a\prec b$ while the definition of $\tau$ and $\rho$ implies that $(a,i)=(a,1) <_\rho (b,0)=(b,j)$. 
Take an arbitrary  $(c,k)\in \aL$ such that 
$(a,1)\leq_\rho(c,k)\leq_\rho (b,0)$. Since 
$(a,1)\leq_\pi(c,k)\leq_\pi (b,0)$ by  \eqref{eqbzKnknSrVd} and  $a\prec b$, we have that $c\in\set{a,b}$. If $c=a$, then we can apply 
\eqref{eqmkZsprmsJlKrNqg} with $x:=a$ to exclude that $k=0$, whence $(c,k)=(a,1)$. Similarly, if $c=b$, then \eqref{eqmkZsprmsJlKrNqg} with $x:=b$ excludes that $k=1$, whereby 
$(c,k)=(b,0)$. Thus $(c,k)$ is necessarily $(a,1)$ or $(b,0)$, and 
it follows that $(a,i)=(a,1)\prho (b,0)=(b,j)$. Hence, \cref{cta} describes a $\rho$-covering, as required. We have seen that the disjunction of \cref{caa}, \cref{cab}, and  \cref{cta} describes $\prho$ correctly.

Next, for later use, we observe the following.
For  covering pairs $e_1\prec f_1$ and $e_2\prec f_2$ in $L$, we say that they are \emph{transposed} if  
$e_2\vee f_1=f_2$ and $e_2\wedge f_1=e_1$, or $e_1\vee f_2=f_1$ and $e_1\wedge f_2=e_2$. Since $T$ is compatible, 
\begin{equation}\left.
\parbox{6.8cm}{if $e_1\prec f_1$ and $e_2\prec f_2$ are transposed edges in $L$, then $e_1\pT f_1\iff e_2\pT e_f$;}\,\,\right\}
\label{pbxTrsPsRsltTrzs}
\end{equation}
referencing this property is how we can exploit the compatibility of $T$.

Next, \eqref{pbxBEZ} formulates the \emph{BEZ Lemma}, named after the initials of its inventors,  Bj\"orner, Edelman, and Ziegler~\cite[Lemma 2.1]{bjorneratall}. As a preparation to it, recall that a poset is \emph{bounded} if it has a (necessarily unique) least element $0$ and a (necessarily unique) greatest element $1$, and a poset is of \emph{finite length} is there is a finite upper bound on the lengths of its chains. 
\begin{equation}\left.
\parbox{8.2cm}{If $P$ is a bounded poset of finite length such that for any $x$ and $y$ in $P$ with a common lower cover the join $x\vee y$ exists, then $P$ is a lattice.}\,\,\right\}
\label{pbxBEZ}
\end{equation}
Since $\pi\subseteq \rho$, it is clear that $\dbl LT$ is a bounded poset with bottom element $(0,0)$ and  top element $(1,1)$. Let 
\begin{equation}
\dots (a_{-1},k_{-1})\lrho (a_{0},k_{0})\lrho (a_{1},k_{1})\lrho\dots 
\label{eqmfHbnzTsq}
\end{equation}
be an arbitrary chain in $\dbl LT$. Since $L$ is of finite length, say, of length $n$, it follows from \eqref{eqbzKnknSrVd} that the set $\set{\dots,a_{-1},a_0,a_1,\dots}$ consists of at most $n+1$ elements. The set $\set{\dots,k_{-1},k_0,k_1,\dots}$ has at most two elements since it is a subset of $\cht$. Hence, the chain in \eqref{eqmfHbnzTsq} consists of at most $2(n+1)$ elements, and we obtain that $\dbl LT$ is of finite length. 
Hence, the BEZ Lemma is applicable. In fact, for later use, we are going to prove a bit more than required by \eqref{pbxBEZ}. By a \emph{covering square} in a poset we mean a quadruple $(o,a,b,i)$ of four distinct elements such that $o\prec a$, $o\prec b$, $a\prec i$, and $b\prec i$.   We claim that 
\begin{equation}\left.
\parbox{7.4cm}{If $x,y,z\in\aL$ such that $x\neq y$, $z\prho x$, and $z\prho y$, then the join $x\vee y$ exists in $\dbl LT$ and $(z,x,y,x\vee y)$ is a covering square in $\dbl LT$.}\,\,\right\}
\label{pbxZjhmKmsTn}
\end{equation}
In order to prove \eqref{pbxZjhmKmsTn}, we have to deal with several cases depending on the position of $z$ and the covering types \cref{caa},\dots, \cref{cta} that occur.

\begin{case}\label{case1} We assume that $z$ is of the form $(c,0)$ and both of $z\prho x$ and $z\prho y$ are \cref{cab}-coverings. Then $x=(a,0)$, $y=(b,0)$, $c\prec a$, $c\prec b$, $c\npT a$, and $c\npT b$. Letting $u:=(a\vee b,0)$, we claim that $u=x\vee y$ and $(z,x,y,u)$ is a covering square in $\dbl LT$. By the modularity of $L$, $a\prec a\vee b$ and $b\prec a\vee b$ in $L$. So, $(c,a,b,a\vee b)$ is a covering square in $L$. By \eqref{pbxTrsPsRsltTrzs}, $a\npT a\vee b$ and 
$b\npT a\vee b$, whereby $x=(a,0)\prec (a\vee b,0)$ and  
$y=(b,0)\prec (a\vee b,0)$ are \cref{cab}-coverings, so we have a covering square in $\dbl LT$. Clearly, $(a\vee b,0)$ is an upper bound of $x=(a,0)$ and $y=(b,0)$, and it follows easily from \eqref{eqbzKnknSrVd} that is the least upper bound. Hence, $x\vee y=(a\vee b,0)$, so the required join exists. This completes Case~\ref{case1}. 
\end{case}

Before the next case, we prove the following auxiliary statement.
\begin{equation}\left.
\parbox{7cm}{If $u\prec v$, $u\npT v$, and $v\leq w$ in $L$, then $(u,1)\rrho (w,0)$ implies that $(v,1)\rrho (w,0)$.
}\right\}
\label{pbxmcshKjqsJ}
\end{equation}
In order to show this, take a shortest repetition-free $(\pi\cup\tau)$-sequence from $(u,1)$ to $(w,0)$. As in \eqref{eqprsrzT}, let $(x_i,k_i)$, $i=0,1,\dots,n$, be the members of this sequence. Since $(x_0,k_0)=(u,1)$ is an upper element but $(x_n,k_n)=(w,0)$ is not, there is a unique integer $t\in\set{0,1,\dots, n-1}$ such that the $(x_i,k_i)$ are
upper elements for $i=0,1,\dots, t$ but $(x_{t+1},k_{t+1})$ is a lower element. That is, $k_0=k_1=\dots k_t=1$ but $k_{t+1}=0$. 
If there is an $i\in\set{0,\dots,t}$ such that $v\leq x_i$, then   $(v,1)\lqrho (x_i,1)$ by $\pi\subseteq \rho$, $(x_i,1)\lqrho (w,0)$ by the second half of the sequence, and so 
the transitivity of $\rho$ yields that $(v,1)\rrho (w,0)$, as required. 
Hence, we can assume that $v\not\leq x_i$ for $i=0,1,\dots, t$.  
The transition from $(x_t,k_t)=(x_t,1)$ to $(x_{t+1},k_{t+1})=
(x_{t+1},0)$ in the sequence is not a $\pi$-step, whence it is a $\tau$-step. This implies that $x_t\pT x_{t+1}$ and, in particular, $x_t\prec x_{t+1}$. Now there are two cases depending on whether $v\leq x_{t+1}$ or not.

First, let $v\leq x_{t+1}$. Since $v\not\leq x_t$ and $x_t\prec x_{t+1}$, it follows from $x_t<x_t\vee v\leq x_{t+1}$ that $x_t\vee v=x_{t+1}$. Similarly, $u\leq x_t$ by  \eqref{eqbzKnknSrVd}, and so  $u\leq x_t\wedge v<v$ and $u\prec v$ yield that $x_t\wedge v=u$. That is, $x_t\prec x_{t+1}$ and $u\prec v$ are transposed edges. But this contradicts \eqref{pbxTrsPsRsltTrzs} since $x_t\pT x_{t+1}$ but $u\npT v$. 

Second, assume that $v\not\leq x_{t+1}$. Since $u\leq x_{t+1}$ by  \eqref{eqbzKnknSrVd}, $u\leq  x_{t+1}\wedge v < v$ and $u\prec v$ give that $x_{t+1}\wedge v=u$. 
By the Isomorphism Theorem for Modular Lattices, see, for example, Gr\"atzer\cite[Theorem 348]{ggfoundbook}, the maps
\begin{equation}\left.
\begin{aligned}
\phi&\colon[u=x_{t+1}\wedge v,x_{t+1}]\to [v,x_{t+1}\vee v],\,\,\text{ defined by }\,\,  s\mapsto s\vee v,\,\,\text{ and}
\cr
\psi&\colon [v,x_{t+1}\vee v] \to [u,x_{t+1}],\,\,\text{  defined by }\,\, s\mapsto s\wedge x_{t+1}
\end{aligned}
\!\!\!\!\right\}
\label{alignMdhSznlt}
\end{equation}
are reciprocal lattice isomorphisms. Hence, using that $x_t\prec x_{t+1}$, we have that
$x_t\vee v=\phi(x_t)\prec \phi(x_{t+1})= x_{t+1}\vee v$. Furthermore, $x_{t+1}\vee x_{t}\vee v = x_{t+1}\vee v$ and 
$x_t= \psi(\phi(x_t))= (x_t\vee v)\wedge x_{t+1}$, showing that $x_t\prec x_{t+1}$ and $x_t\vee v\prec  x_{t+1}\vee v$ are transposed edges. 
This fact, $x_t\pT x_{t+1}$, and \eqref{pbxTrsPsRsltTrzs} give that $x_t\vee v\pT x_{t+1}\vee v$. Hence, $(x_t\vee v,1) \rtau (x_{t+1}\vee v,0)$, which gives that $(x_t\vee v,1) \lqrho (x_{t+1}\vee v,0)$. Since $\pi\subseteq \rho$, we have that $(v,1)\lqrho (x_t\vee v,1)$.
Since $(x_{t+1},0)$ occurs in our sequence, $x_{t+1}\leq w$ by   \eqref{eqbzKnknSrVd}. Hence, using the assumption $v\leq w$, we have that $x_{t+1}\vee v\leq w$ and we conclude from $\pi\subseteq \rho$ that $(x_{t+1}\vee v,0)\lqrho (w,0)$. The last three inequalities with $\lqrho$ yield that $(v,1)\lqrho (w,0)$, in other words, $(v,1)\rrho (w,0)$, proving \eqref{pbxmcshKjqsJ}. 

Note that although the use of modularity could have been avoided at several places in our considerations, it seems to be important at \eqref{alignMdhSznlt}.

\begin{case}\label{case2} We assume that $z=(c,0)$ is a lower element and at least one of $z\prho x$ and $z\prho y$ is a  \cref{caa}-covering. The other covering is necessarily a \cref{cab}-covering, so we can assume that $x=(c,1)$ and $y=(b,0)$ with $c\prec b$ and $c\npT b$. We claim that $(b,1)$ is the join of $x$ and $y$; it is clearly an upper bound. Let $(d,k)$ be an arbitrary upper bound of $x=(c,1)$ and $y=(b,0)$. By \eqref{eqbzKnknSrVd}, $b\leq d$. 
If $k=1$, then $(b,1)\lqrho (d,1)=(d,k)$ by $\pi\subseteq \rho$, as required. Thus, we assume that $k=0$. Letting $(c,b,d)$ play the role of $(u,v,w)$, we obtain from \eqref{pbxmcshKjqsJ} that $(b,1)\lqrho (d,0)=(d,k)$, showing that $(b,1)$ is the join of $x=(c,1)$ and $y=(b,0)$ in $\dbl LT$. Since $c\npT b$ and $c\prec b$, we have that $x=(c,1)\prho (b,1)$ is a \cref{cab}-covering. Since $y=(b,0)\prho (b,1)$ is a \cref{caa}-covering, $(z,x,y,x\vee y)$ is a covering square in $\dbl LT$, completing Case~\ref{case2}.
\end{case}

\begin{case}\label{case3} We assume that $z=(c,1)$ is an upper element and at least one of the coverings $z\prho x$ and $z\prho y$ is a \cref{cta}-covering. Let $z\prho x$ such a covering, that is, 
$x=(a,0)$ such that $c\pT a$ (and so $c\prec a$). We obtain from \eqref{eqtxtblTBlnl} and $y\neq x$ that $z\prho y$ cannot be a \hbox{\cref{cta}}-covering, whereby it is a \cref{cab} covering, that is, 
$y=(b,1)$ with $c\prec b$ and $c\npT b$. By the (upper semi-)modularity of $L$, $(c,a,b,a\vee b)$ is a covering square in $L$. This fact and \eqref{pbxTrsPsRsltTrzs} give that $b\pT a\vee b$ but $a\npT a\vee b$. Thus, $y=(b,1)\prho (a\vee b,0)$ is a \cref{cta}-covering and  $x=(a,0)\prho (a\vee b,0)$ is a \cref{cab}-covering. We have seen that $(z,x,y,(a\vee b,0))$ is a covering square in $\dbl LT$. In particular, $(a\vee b,0)$ is an upper bound of $x=(a,0)$ and $y=(b,1)$. Let $(d,k)$ be another upper bound. Since $a\vee b\leq d$ by  \eqref{eqbzKnknSrVd}, we obtain from $\pi\subseteq \rho$ that $(a\vee b,0)\leq (d,k)$, as required. This completes Case~\ref{case3}.
\end{case}

\begin{case}\label{case4} We assume that $z=(c,1)$ is an upper element and none of the coverings $z\prho x$ and $z\prho y$ is a \cref{cta}-covering. Then both are  \cref{cab}-coverings, so 
$x=(a,1)$, $y=(b,1)$, $c\prec a$, $c\prec b$, $c\npT a$, and $c\npT b$. By the modularity of $L$, $(c,a,b,a\vee b)$ is a covering square, and it follows from \eqref{pbxTrsPsRsltTrzs} that $a\npT a\vee b$ and $b\npT a\vee b$. Hence, $(z,x,y,(a\vee b,1))$ is a covering square in $\dbl LT$ with all of its edges being \cref{cab}-coverings. This square shows that  $(a\vee b,1)$ is an upper bound of $x$ and $y$. Let $(d,k)$ be an arbitrary upper bound of $x=(a,1)$ and $y=(b,1)$ in $\dbl LT$. From \eqref{eqbzKnknSrVd}, it follows that $a\vee b\leq d$. Observe that $\pi\subseteq \rho$ gives that $(a\vee b,1)\lqrho (d,1)$. So we can assume that $k=0$ since otherwise $(a\vee b,1)\lqrho (d,k)$ has already been shown.   
Since  $(a,1)\lqrho (d,0)$, $a\prec a\vee b$, $a\npT a\vee b$, and $a\vee b\leq d$, we can apply \eqref{pbxmcshKjqsJ} with $(a,a\vee b, d)$ playing the role of $(u,v,w)$ to conclude that $(a\vee b,1)\lqrho (d,0)=(d,k)$. 
This completes Case~\ref{case4}. 
\end{case}

Cases~\ref{case1}--\ref{case4} prove the validity of \eqref{pbxZjhmKmsTn}, and so $\dbl LT$ is a lattice by \eqref{pbxBEZ}.  The following statement will be used to prove that $\dbl LT$ is modular. We claim that 
\begin{equation}\left.
\parbox{9cm}{If $x,y,z\in\aL$ such that $x\neq y$, $x\prho z$, and $y\prho z$, then  $(x\wedge y,x,y,z)$ is a covering square in $\dbl LT$.}\,\,\right\}
\label{pbxgjhswTszhtD}
\end{equation}
Now that we already know that $\dbl LT$ is a lattice, the proof of \eqref{pbxgjhswTszhtD} is easier than that of \eqref{pbxZjhmKmsTn}. Indeed, it suffices to show that $x$ and $y$ from \eqref{pbxgjhswTszhtD} have a common lower cover. However, we have to deal with several cases again. First, assume that at least one $x$ and $y$, let it be $x$, is \cref{cta}-covered by $z$. Hence $z=(c,0)$, $x=(a,1)$, and $a\pT c$ (and so  $a\prec c$). By \eqref{eqtxtblTBlnl}, $x$ is the only lower \cref{cta}-cover of $x$, whereby $y\prho z$ is a \cref{cab}-covering, $y$ is of the form $y=(b,0)$ with $b\prec c$ and $b\npT c$. By the modularity of $L$, $(a\wedge b, a,b,c)$ is a covering square in $L$. Hence, using \eqref{pbxTrsPsRsltTrzs}, we obtain the $a\wedge b\pT b$ and $a\wedge b\npT a$. Therefore, 
$(a\wedge b,1)\prho (b,0)=y$ is a \cref{cta}-covering while 
$(a\wedge b,1)\prho (a,1)=x$ is a  \cref{cab}-covering, showing that $x$ and $y$ have a common lower cover, as required. Thus, in the rest of the cases, we can disregard the situation when $x\prho z$ or $y\prho z$ is a \cref{cta}-covering. 

Second, assume that  both $x\prec z$ and $y\prec z$ are \cref{cab}-coverings. Then $z=(c,i)$, $x=(a,i)$, $y=(b,i)$, $a\prec c$, $b\prec c$, $a\npT c$, and $b\npT c$. The modularity of $L$ yields that $(a\wedge b, a,b,c)$ is a covering square in $L$. Hence, $a\wedge b\npT a$ and 
$a\wedge b\npT b$ by \eqref{pbxTrsPsRsltTrzs}, and it follows that $(a\wedge b,i)$ is \cref{cab}-covered both by $x=(a,i)$ and $y=(b,i)$.  So $x$ and $y$ has a common lower cover in this case.

Third, since no element can have two distinct lower \cref{caa}-covers, there remains only one case: 
one of $x\prho z$ and $y\prho z$ is a  \cref{caa}-covering and the other one is a \cref{cab}-covering.
Hence, we can assume that $z=(c,1)$, $x=(a,1)$ with $a\prec c$ but $a\npT c$, and $y=(c,0)$. Clearly, $(a,0)\prho (a,1)=x$ is a \cref{caa}-covering while $(a,0)\prho (c,0)=y$ is a \cref{cab}-covering, showing that $x$ and $y$ have a common lower cover again. This proves \eqref{pbxgjhswTszhtD}.

\begin{figure}[htb] 
\centerline
{\includegraphics[scale=1.0]{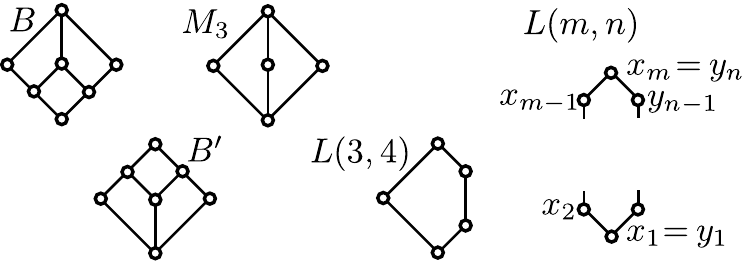}}
\caption{The lattices in Jakub\'\i{}k's theorem; see \eqref{pbxJakubik}
\label{figurc}}
\end{figure}%

If $L_1$ is a sublattice of another lattice $L_2$ such that every covering pair $a\prec_{L_1} b$ is also a covering pair in $L_2$, then $L_1$ is a \emph{cover-preserving sublattice} of $L_2$. However, instead of saying that $M_3$ is  cover-preserving sublattice of $L$, we usually say shortly that \emph{$L$ has a covering $M_3$}.  
The lattices $B$, $B'$, $M_3$, and $L(m,n)$ (for $m\geq 3$ and $n\geq 4$) are given in Figure~\ref{figurc}. It is proved in Jakub\'{\i}k  \cite[Theorems 1 and 2]{jakubik} that for a lattice $L$ of finite length (in particular, for a finite lattice $L$),
\begin{equation}\left.
\parbox{8.2cm}{$L$ is modular if and only if none of the lattices $B$, $B'$, and $L(m,n)$ ($m\geq 3$, $n\geq 4$) is a cover-preserving sublattice of $L$. Furthermore, $L$ is distributive if and only if it is modular and it has no covering $M_3$.}
\,\,\right\}
\label{pbxJakubik}
\end{equation}
Observe that each of the $L(m,n)$  ($m\geq 3$, $n\geq 4$) and $B$ has three elements, $x,y,z$, such that $x\prec z$, $y\prec z$, $x\neq y$, but $(x\wedge y, x, y, z)$ is not a covering square. It follows from \eqref{pbxgjhswTszhtD} that these lattices cannot be cover-preserving sublattices of $\dbl LT$. 
Similarly, $B'$, the dual of $B$, has elements $x,y,z$ such that $z\prec x$, $z\prec y$, $x\neq y$, but $(z,x,y,x\vee y)$ is not a covering square in $B'$. Hence, $B'$ is not a cover-preserving sublattice of $\dbl LT$ by \eqref{pbxZjhmKmsTn}. Therefore, \eqref{pbxJakubik} yields that $\dbl LT$ is modular.

Before dealing with the distributive case, the following auxiliary statement is worth proving.
\begin{equation}\left.
\parbox{6.7cm}{If $(d_1,i)\prho (d_2,i)$ and $(d_1,i)\prho (d_3,i)$ are \cref{cab}-coverings, then $(d_2,i)\vee (d_3,i)=(d_2\vee d_3,i)$ holds in $\dbl LT$.}
\,\,\right\}
\label{pbxZhBfrC}
\end{equation}
Indeed, if the premise of \eqref{pbxZhBfrC} holds, then 
$d_1\prec d_2$ and $d_1\prec d_3$, so  modularity yields that  $(d_1, d_2, d_3, d_2\vee d_3)$ is a covering square in $L$. Since $d_1\npT d_2$ and $d_1\npT d_3$, we obtain from \eqref{pbxTrsPsRsltTrzs} that $d_2\npT d_2\vee d_3$ and $d_3\npT d_2\vee d_3$. Hence, $(d_2,i)\prho (d_2\vee d_3,i)$ and  
$(d_3,i)\prho (d_2\vee d_3,i)$, since they are \cref{cab}-coverings. We are in a lattice, so these two coverings yield the validity of \eqref{pbxZhBfrC}. 

Next, assume that $L$ is distributive;  we already know that $\dbl LT$ is modular. For the sake of contradiction, suppose that $\dbl LT$ is not distributive. By \eqref{pbxJakubik}, $\dbl LT$ has a covering $M_3=\set{u,x,y,z,v}$, where $u\prho x\prho v$, $u\prho y\prho v$, and $u\prho z\prho v$. There are three cases to consider. 

First, assume that each of $u\prho x$, $u\prho y$, and $u\prho z$ is a \cref{cab}-covering. Then we can write that $u=(e,i)$, 
$x=(a,i)$, $y=(b,i)$, and $z=(c,i)$, and it follows from \eqref{pbxZhBfrC} that each of $(a\vee b,i)$, $(a\vee c,i)$ and $(b\vee c,i)$ equals $v$. Hence, $a\vee b=a\vee c=b\vee c$, yielding that  $(e,a,b,c,a\vee b)$ is a (covering) $M_3$ in $L$, contradicting the distributivity of $L$. 

Second, assume that  
least one of  $u\prho x$, $u\prho y$, and $u\prho z$ is a \cref{caa}-covering. Then $u$ is a lower element of the form $u=(e,0)$, whence none of $u\prho x$, $u\prho y$, and $u\prho z$ is a \cref{cta}-covering. Since  at most one of these three coverings can be a  \cref{caa}-covering, we can assume that $u=(e,0)\prho (e,1)=x$ is a \cref{caa}-covering while $u\prho (b,0)=y$ and $u\prho (c,0)=z$ are  \cref{cab}-coverings.   By \eqref{pbxZhBfrC}, $v=y\vee z=(b\vee c,0)$. Hence, $(e,1)=x\prho v=(b\vee c,0)$ is a covering, because we are in a covering $M_3$. The only way that a lower element can cover an upper one is a \cref{cta}-covering. Hence,  $(e,1)\prho (b\vee c,0)$ is a \cref{cta}-covering, $e\pT b\vee c$ and, in particular, $e\prec b\vee c$. But this is a contradiction, because the \cref{cab}-coverings $(e,0)=u \prho  y=(b,0)$ and 
$(b,0)=y\prho v= (b\vee c,0)$ give that  $e < b < b\vee c$.

Third, assume that  least one of  
$u\prho x$, $u\prho y$, and $u\prho z$ is a \cref{cta}-covering. Then $u$ is an upper element of the form $u=(e,1)$; let $u\prho x$ a \cref{cta}-covering. Then $x$ is of the form $x=(a,0)$ with $e\pT a$ and, in particular, $e\prec a$.
It follows from \eqref{eqtxtblTBlnl} that none of $u\prho y$, and $u\prho z$ is a \cref{cta}-covering, and they are not \cref{caa}-coverings because $u$ is an upper element. So 
 $u\prho y$ and $u\prho z$ are  \cref{cab}-coverings and we can write that $y=(b,1)$ and $z=(c,1)$ with $b\neq c$.
From \eqref{pbxZhBfrC}, we obtain that $v=y\vee z=(b\vee c,1)$.
Hence, $(a,0)=x\prho v=(b\vee c,1)$. This is neither a \cref{cab}-covering, nor a \cref{cta}-covering, because $x$ is a lower element and $v$ is an upper one. Thus, $(a,0)\prho(b\vee c,1)$ is a \cref{caa}-covering and so $a=b\vee c$. Hence, $e<b<b\vee c=a$ in $L$, contradicting $e\prec a$. 

Now that all the three cases have led to contradiction, we have shown that the modular lattice $\dbl LT$ has no covering $M_3$, and it follows from \eqref{pbxJakubik} that this lattice  is distributive. 
The proof of Theorem~\ref{thmdoubling} is complete.
\end{proof}

\section{An application of doubling tolerances to coalition lattices}\label{sectionfincoal}

For a finite poset $\alg P=\poset P$, the set of all subsets of $P$ will be denoted by $\Pow P$ or $\Pow{\alg P}$. Note at this point that the relations $<$ and $\leq$ in a poset $\alg P$ mutually determine each other; this allows us to use both of them even if only one is given originally. 

\begin{definition}[Cz\'edli and Poll\'ak~\cite{czgpgy}]
Let $\alg P=\poset P$  be a finite poset. For $X,Y\in \PowP$, a map (function) $\phi\colon X\to Y$ is  \emph{extensive} if $\phi$ is injective and $x\leq \phi(x)$ holds for every $x\in X$. Let 
$X\leq Y$ mean that there exists an extensive map $X\to Y$.  With this meaning of ``$\leq$'', the poset $\ieCoalP$  is the \emph{coalition poset} of $\alg P$, and its elements are called the \emph{coalitions} of $\alg P$.
\end{definition}

When we consider a subset $X$ of $P$ as a member of $\CoalP$, then we call it a coalition of $\alg P$ rather than a subset. 
Note that since 1995, when \cite{czgpgy} was published, the terms ``coalition'' and  ``coalition lattice'' have also been used with different meanings in mathematics and informatics; see, e.g., \cite{meijedeatal}, \cite{nakanoatall}, and \cite{sandholmatal}. 

The Hasse diagram of our finite poset  $\alg P=\poset P$ is also a graph; the (connectivity) components of this graph are the \emph{components} of $\alg P$. These components are also posets with the orderings restricted from  $\alg P$ to them.
We say that $\alg P$ is \emph{upper bound free} if no two incomparable elements has an upper bound in $\alg P$. \emph{Lower bound free} posets are defined dually. Note that $\alg P$ is both upper bound free and lower bound free if and only if all of its components are chains. In \cite{czgpgy}, we proved that

\begin{proposition}[Cz\'edli and Poll\'ak~\cite{czgpgy}]\label{propczgpgy} Let $\alg P=\poset P$ be a finite poset.
\begin{enumeratei}
\item\label{propczgpgya} $\CoalP$ is a lattice if and only if $\alg P$ is upper bound free.
\item\label{propczgpgyb} If $\alg P_1$, \dots , $\alg P_n$ is the list of the components of $\alg P$, then the lattice $\CoalP$ is (isomorphic to) the direct product $\prod_{i=1}^k{\iCoalP i}$ 
\item\label{propczgpgyc} If $\CoalP$ is a lattice and $\alg P$ is lower bound free, then $\CoalP$ is distributive.
\item\label{propczgpgyd} If  $\CoalP$ is a distributive lattice, then $\alg P$ is lower bound free.
\end{enumeratei}
\end{proposition}

Next, we formulate a particular case of  \ref{propczgpgy}\eqref{propczgpgyc},  which we are going to prove here with the help of doubling tolerances; note that the   
conjunction of this particular case with (the more or less trivial) \ref{propczgpgy}\eqref{propczgpgyb} implies  \ref{propczgpgy}\eqref{propczgpgyc}.

\begin{corollary}\label{corHnRjsB}
If $\alg P=\poset P$ is a finite chain, then $\CoalP$ is a distributive lattice.
\end{corollary}

We call the statement above a corollary because there will be no separate proof of it; this corollary will prompt follow from the following lemma, the main achievement of (the current) Section~\ref{sectionfincoal}. The primary purpose of this lemma is to present an example and an application of  our doubling construction.

\begin{lemma}\label{lemmazBfmWvVjGtRx}
Let $\alg P=\poset P$ be a finite non-singleton chain with 
smallest element $0$, and let $w$ denote its unique atom. Also, let $\alg P'$ be its principal filter $\filter w$; that is, $P'=P\setminus\set 0$. On the lattice $\pCoalP$, we define a relation $T$ as follows: 
\begin{equation}
T:=\set{(A,B)\in \pCoalP^2: A\cup \set w=B\cup\set w }.
\label{eqTdFlsR}
\end{equation}
Then $T$ is a doubling tolerance on $\pCoalP$ and
$\dbl{\pCoalP}T$ is isomorphic to $\CoalP$. Also, $T$ is a congruence and both $\pCoalP$ and $\CoalP$ are distributive lattices.
\end{lemma}

Next, we state and prove some lemmas that will be needed in the proof of Lemma~\ref{lemmazBfmWvVjGtRx}; these lemmas can be of separate interest.

\begin{lemma}\label{lemmarnDrnjKszBl} 
If $\alg P$ is a finite chain and $A_1,A_2\in\CoalP$, then 
\[
A_1\cap A_2\subseteq A_1\wedge A_2\subseteq A_1\cup A_2
\,\,\text{and }\,\,
A_1\cap A_2\subseteq A_1\vee A_2\subseteq A_1\cup A_2.
\]
\end{lemma}

\begin{proof} First, 
with the notation given in 
Lemma~\ref{lemmarnDrnjKszBl}, let $A_1,A_2\in \CoalP$ be \emph{nonempty} coalitions. Let $c$ be the largest element of the  nonempty set
\[
H:=\set{x_1\wedge x_2: x_1\in A_1\text{ and }x_2\in A_2}.
\]
Since $\alg P$ is a chain, $c$ exists,  $H\subseteq A_1\cup A_2$, and, in particular, $c\in A_1\cup A_2$. Pick $b_1\in A_1$ and $b_2\in A_2$ such that $c=b_1\wedge b_2$.
For $i\in\set{1,2}$, let $a_i=c$ if $c\in A_i$, and let $a_i=b_i$ otherwise. Then $a_1\in A_1$, $a_2\in A_2$, and $c=a_1\wedge a_2$. Let $A_i':=A_i\setminus\set {a_i}$ for $=1,2$, and let $\widehat  P:=P\setminus\set c$. Then $\poset{\widehat P}=:\widehat{\alg P}$ is a subchain of $\alg P$.
For $i\in\set{1,2}$, we have that $c\notin A'_i$, since otherwise $c\in A_i'\subseteq A_i$ would give that $a_i=c$ and so $a_i=c\in A_i'=A_i\setminus\set{a_i}$ would be a contradiction. Hence, and $A_1',A_2'\in \whCoalP$. Let $C'$ be the meet of $A_1'$ and $A_2'$ in $\whCoalP$. As a particular case of   Cz\'edli and Poll\'ak~\cite[Proposition 1]{czgpgy}, 
\begin{equation}
A_1\wedge A_2=C'\cup\set{c}\text{, that is, }
A_1\wedge_{\CoalP}A_2 = (A_1'\wedge_{\whCoalP} A_2')\cup\set c.
\label{eqmTszvSzkzmWjTl}
\end{equation}
Next, we claim that 
\begin{equation}
\text{for any $A_1,A_2\in \CoalP$, we have that $A_1\wedge A_2\subseteq A_1\cup A_2$.}
\label{eqtxtMtrszgstsNk}
\end{equation}
We prove this by induction on $|P|$. If $\alg P$ is a singleton, then \eqref{eqtxtMtrszgstsNk} is clear.
If $A_1$ or $A_2$ is the empty coalition, then so is $A_1\wedge A_2$ and \eqref{eqtxtMtrszgstsNk} is clear again. So, for the induction step, we can assume that $|P|>1$, \eqref{eqtxtMtrszgstsNk} holds for smaller chains, and none of $A_1$ and $A_2$ is empty. With the notation used in \eqref{eqmTszvSzkzmWjTl}, $c\in A_1\cup A_2$. By the induction hypothesis, $C'=A_1'\wedge_{\whCoalP} A_2'\subseteq A_1'\cup A_2'\subseteq A_1\cup A_2$. Hence, \eqref{eqmTszvSzkzmWjTl} implies that $A_1\wedge A_2\subseteq A_1\cup A_2$, proving \eqref{eqtxtMtrszgstsNk}.

Since the map $\delta\colon \CoalP\to\CoalP$, defined by $A\mapsto P\setminus A$, is a dual lattice automorphism by  Cz\'edli and Poll\'ak~\cite[Proposition 2]{czgpgy} and $\delta$ is an involution, we obtain  that $A_1\vee A_2=\delta(\delta(A_1)\wedge \delta(A_2))$. Note that  $\delta$ is also a dual automorphism of the powerset lattice $(\Pow P;\subseteq)$ by the de Morgan laws. Hence, letting $X_i=\delta(A_i)$, applying  \eqref{eqtxtMtrszgstsNk} for $X_1$ and $X_2$, and using that $\delta$ is an involution, we obtain that for any $A_1,A_2\subseteq \CoalP$, 
\begin{equation}
A_1\vee A_2=\delta(X_1\wedge X_2)\overset{\eqref{eqtxtMtrszgstsNk}}\supseteq \delta(X_1\cup X_2)=\delta(X_1)\cap\delta(X_2)=A_1\cap A_2.
\label{eqzrdPfTfsStmMgr}
\end{equation}
It has been proved in Cz\'edli~\cite[displays in page 102]{czghorncoal} that $A_1\cap A_2\subseteq A_1\wedge A_2$ and $A_2\vee A_2\subseteq A_1\cup A_2$, even without assuming that $\alg P$ is a chain. Combining these inequalities with \eqref{eqtxtMtrszgstsNk} and 
\eqref{eqzrdPfTfsStmMgr}, we obtain the statement of the lemma.
\end{proof}

Note that if $\alg V$ is the three-element meet-semilattice that is not a lattice, then 
$\Coal{\alg V}$ is a lattice but
there are singleton coalitions $A_0,A_1,A_2\in\Coal{\alg V}$ such that $A_1\wedge A_2=A_0\not\subseteq A_1\cup A_2$. Hence, the assumption that $\alg P$ is a chain is essential in Lemma~\ref{lemmarnDrnjKszBl}.

The following statement is taken from Cz\'edli~\cite[Lemma 1]{czgJH}.

\begin{lemma}[{Cz\'edli~\cite[Lemma 1]{czgJH}}]\label{lemmaskzsRsznmCsnsT}
If $X\leq Y$ in a coalition lattice, then there exists an extensive map $X\to Y$ that acts identically on $X\cap Y$. 
\end{lemma}

The \emph{height} $\height(x)$ of an element $x$ of a chain $\alg P$ is defined in the usual way: $\height(x)=k$ if and only if $\ideal x:=\set{y\in P: y\leq x}$ consists of $k+1$ elements. 
The \emph{strength} $\str X$ of a coalition $X$ is defined to be 
\[
\str X:=|X|+\sum_{u\in X} \height(u) = \sum_{u\in X} |\ideal u|.
\]
Note that for $X\leq Y$ in $\CoalP$, we have that $\str X\leq \str Y$. With these concepts,  we can describe the covering relation in $\CoalP$ as follows. Note that, as opposed to some parts of mathematics (far from lattice theory), here $A\subset B$ means the conjunction of $A\neq B$ and $A\subseteq B$.

\begin{lemma}\label{lemmaczhGrMqsnh} Let $\alg P=\poset P$ be a finite chain with $|P|\geq 2$, its smallest element is denoted by $0$. Let $A,B\in \CoalP$ such that $A<B$. Then $A\prec B$ in $\CoalP$ if and only if the following two conditions hold.
\begin{enumeratei}
\item\label{lemmaczhGrMqsnha} $\str B=\str A+1$;
\item\label{lemmaczhGrMqsnhb} either $A\subset B$, or $|A|=|B|=1+|A\cap B|$.
\end{enumeratei}
Furthermore, for later reference, we note that 
\begin{enumeratei}\setcounter{enumi}{2}
\item\label{lemmaczhGrMqsnhc} if $A\prec B$ and $A\subset B$, then $B=A\cup\set 0$;
\item\label{lemmaczhGrMqsnhd} if $A\prec B$ and $|A|<|B|$, then $B=A\cup\set 0$.
\end{enumeratei}
\end{lemma}

\begin{proof}
Observe that 
\begin{equation}
\text{\eqref{lemmaczhGrMqsnha} and \eqref{lemmaczhGrMqsnhb} together imply that $A<B$ in $\CoalP$.}
\label{eqtxtHmTszGnGkRs}
\end{equation}
Indeed, if $A\subset B$, then $A<B$ is obvious. Assume that  $|A|=|B|=1+|A\cap B|$, then $A=X\cup\set a$ and $B=X\cup\set b$ with $\set{a,b}\cap X=\emptyset$ and $a\neq b$. It follows from \eqref{lemmaczhGrMqsnha} that $\height(a)<\height(b)$ and so $a<b$. Thus, $\set{(a,b)}\cup\idmap X\colon A\to B$ is an extensive map, and so $A\leq B$. But $A\neq B$ by  \eqref{lemmaczhGrMqsnha}, and we conclude that $A<B$, as required. This proves \eqref{eqtxtHmTszGnGkRs}. 
Our next observation is that 
\begin{equation}
\text{if $A,B\in \CoalP$, then $A<B$ implies that $\str A<\str B$.}
\label{eqtxtkcsNTszGfhrJ}
\end{equation}
In order to see this, assume that $A<B$. Pick an extensive map $\phi\colon A\to B$. 
Since $A\neq B$,  either $\phi$ is not surjective, or $x<\phi(x)$ for some $x\in A$, in addition to $(\forall y\in A)(y\leq \phi(y))$, whereby $\str A<\str B$ follows easily, proving 
\eqref{eqtxtkcsNTszGfhrJ}.
Combining \eqref{eqtxtHmTszGnGkRs} and \eqref{eqtxtkcsNTszGfhrJ}, we obtain immediately that 
the conjunction of \eqref{lemmaczhGrMqsnha} and \eqref{lemmaczhGrMqsnhb} implies that $A\prec B$ in $\CoalP$.

Next, assume that $A\prec B$. We are going to prove 
that   \eqref{lemmaczhGrMqsnha} and \eqref{lemmaczhGrMqsnhb} hold. There are two subcases, depending on $A\subseteq B$ or $A\not\subseteq B$. First we deal with the case $A\subseteq B$. 
Since $A\neq B$, we have that $A\subset B$. In particular, we have already obtained that \eqref{lemmaczhGrMqsnhb}  holds.  
If $B\setminus A$ had two distinct elements, $x$ and $y$, then $A\subset A\cup\set{x}\subset A\cup\set{x,y}\subseteq B$ would give that  $A< A\cup\set{x}< A\cup\set{x,y}\leq B$, contradicting $A\prec B$. Hence, $B\setminus A$ is a singleton $\set{u}$, that is, $B=A\cup\set u$. We claim that $u=0$. Suppose the contrary. Then  $A\cup\set 0 < A\cup\set u=B$, either because $0\in A$ and $A\cup\set 0=A\subset  A\cup\set u=B$, or because $0\notin A$ and the extensive bijection $\idmap A\cup\set{(0,u)}\colon A\cup\set 0\to A\cup\set u=B$ is not the identity map since $0\neq u$. 
Hence, $A\leq A\cup\set{0}<A\cup\set{u}=B$, where the first inequality  must be an equality since $A\prec B$. Thus, $0\in A$ and so $A\neq B\setminus\set 0$.  Also, $0\in B$, because $A\subset B$, whence $B\setminus\set 0\neq B$, whereby $B\setminus\set 0\subset B$ and $B\setminus\set 0< B$.
Since the map $A\to B\setminus\set 0$,
defined by $0\to u$ and $x\mapsto x$ for $x\in A\setminus\set 0$,  is extensive, we obtain that $A\leq B\setminus 
\set 0$. In fact, $A< B\setminus 
\set 0$ since $A\neq B\setminus\set 0$. This inequality together with  $B\setminus\set 0< B$ contradict 
$A\prec B$. Therefore, $u=0$, 
$B=A\cup \set 0\supset A$, and $\str B=\str A+\height(0)+1=\str A +1$. Thus, \ref{lemmaczhGrMqsnh}\eqref{lemmaczhGrMqsnha} and \ref{lemmaczhGrMqsnh}\eqref{lemmaczhGrMqsnhb} hold, as required. 

Second, still assuming that $A\prec B$, we deal with the case $A\not\subseteq B$, that is  $A\setminus B\neq \emptyset$. 
We write $A\setminus B$ in the form $A\setminus B=\set{a_1,a_2,\dots,a_k}$ where $t\geq 1$ and the elements $a_1$, \dots, $a_t$ are pairwise distinct. Choose an extension map $\phi\colon A\to B$ according to Lemma~\ref{lemmaskzsRsznmCsnsT}. 
Since $\phi$ acts identically on $X:= A\cap B$ and $\phi$ is injective, we have that $a_1\leq \phi(a_1)=:b_1$, \dots,  $a_t\leq \phi(a_t)=:b_t$ are outside $A\cap B$, so they are in $B\setminus  A$. Since $a_i\notin B$ but $b_i\in B$, we obtain that 
\begin{equation}
\text{$a_i<b_i$, \ for $i=1,\dots,t$.}
\label{eqtxtdznBgthRmn}
\end{equation} 
With reference to the injectivity of $\phi$ again, we obtain that the elements $b_1$, \dots, $b_t$ are pairwise distinct.
Using the extensive maps
\begin{align*}
&\restrict\phi{X}\cup\set{(a_1,b_1),(a_2,a_2),(a_3,a_3),\dots,(a_t,a_t)},\cr
&\restrict\phi{X}\cup\set{(b_1,b_1),(a_2,b_2),(a_3,a_3),\dots ,(a_t,a_t)}, \cr
&\dots\cr
&\restrict\phi{X}\cup\set{(b_1,b_1),(b_2,b_2),\dots,(b_{t-1},b_{t-1}),(a_t,b_t)},
\end{align*}
and the fact that $X\cup\set{b_1,b_2,\dots,b_t}\subseteq B$ implies that $X\cup\set{b_1,b_2,\dots,b_t}\leq B$, we obtain that 
\begin{equation}\left.
\begin{aligned}
A&=X\cup\set{a_1,\dots,a_t} <
X\cup\set{b_1,a_2,\dots,a_t}\cr
&<
X\cup\set{b_1,b_2,a_3,\dots,a_t}
<X\cup\set{b_1,b_2,b_3, a_3,\dots,a_t}\cr
&<\dots<
X\cup\set{b_1,b_2,\dots,b_t}\leq B.
\end{aligned}
\right\}
\label{eqbzTjsGtttstT}
\end{equation}
Now we are in the position to conclude from  $A\prec B$ and \eqref{eqbzTjsGtttstT} that $t=1$. In order to ease the notation, we let $a:=a_1$ and $b:=b_1$. 
Tailoring \eqref{eqbzTjsGtttstT} to this new notation and $t=1$, we have that 
$A=X\cup\set a<X\cup\set b\leq B$. Hence, taking $A\prec B$ into account, $B=X\cup \set b$. Using \eqref{eqtxtdznBgthRmn}, we can summarize the situation as follows.
\begin{equation}
A=X\cup\set a,\quad B=X\cup \set b,\quad X=A\cap B,\quad\text{and}\quad a<b.
\label{eqvScrlngst}
\end{equation} 
Clearly, \ref{lemmaczhGrMqsnh}\eqref{lemmaczhGrMqsnhb} is an immediate consequence of \eqref{eqvScrlngst}.
We are going to show that 
$\height(b)=\height(a)+1$, because then  \ref{lemmaczhGrMqsnh}\eqref{lemmaczhGrMqsnha} will automatically follow from \eqref{eqvScrlngst}.  
For the sake of contradiction, suppose that $\height(b)\neq \height(a)+1$. Then  $a<b$ yields that $\height(b)\geq \height(a)+2$. If the whole interval $[a,b]$  is disjoint from $X$, then we can pick an element $y$ such that $a< y < b$ (that is,  $\height(a)<\height(y)<\height(b)$ since $\alg P$ is a chain); this $y$ is not in $X$ and, witnessed by  straightforward extensive functions extending $\idmap X$, we have that $A=X\cup\set a<X\cup\set y<X\cup\set b=B$, contradicting $A\prec B$. Hence, $X\cap[a,b]\neq\emptyset$, and so there is a unique smallest element $z\in X$ such that $a < z < b$. 
Let $C:=(A\setminus\set z)\cup \set b= (X\setminus\set z)\cup \set{a,b}=(B\setminus\set z)\cup\set a$. Clearly, $z\in X=A\cap B$ and $z\notin C$ give that $A\neq C\neq B$. The extension functions
\begin{align*}
&A\to C\text{, defined by }
x\mapsto
\begin{cases}
b,&\text{if }x=z,\cr
x,&\text{otherwise}
\end{cases}\quad \text{and}\cr
&
C\to B\text{, defined by }
x\mapsto
\begin{cases}
z,&\text{if }x=a,\cr
x,&\text{otherwise,}
\end{cases}
\end{align*}
see Figure~\ref{figurd} for an illustration, indicate that 
$A <  C < B$, contradicting $A\prec B$. 
This shows that $\height(b)=\height(a)+1$. Thus, with the exception of \eqref{lemmaczhGrMqsnhc} and  \eqref{lemmaczhGrMqsnhd}, the lemma is proved.

\begin{figure}[htb] 
\centerline
{\includegraphics[scale=1.0]{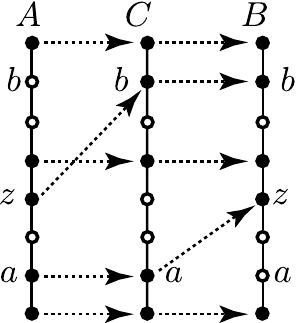}}
\caption{Illustration for $\height(b)\geq \height(a)+2$; $A$, $B$, and $C$ consist of the black-filled elements
\label{figurd}}
\end{figure}%

Next, to prove \eqref{lemmaczhGrMqsnhc}, assume that 
$A\prec B$ and $A\subset B$. We have already proved that \eqref{lemmaczhGrMqsnha} holds. We can write $B$ in the form $A\cup\set{c_1,\dots, c_k}$, where $k\geq 1$ and $A\cap\set{c_1,\dots, c_k}=\emptyset$. Using \eqref{lemmaczhGrMqsnha}, we have that 
$\str A+1=\str B=\str A+k+\height(c_1)+\dots+\height(c_k)$, which implies that $k=1$ and $\height(c_1)=0$, that is, $c_1=0$. Thus, $B=A\cup \set 0$, as required. Therefore, \eqref{lemmaczhGrMqsnhc} holds.

Finally, to prove \eqref{lemmaczhGrMqsnhd}, assume that $A\prec B$ and $|A|<|B|$. Pick an extensive map $\phi\colon A\to B$. Since $|A|<|B|$ and $\phi$ is injective, $\phi(A)\subset B$, whereby $\phi(A)<B$. Hence, $A\leq \phi(A)<B$ and $A\prec B$ imply that $A=\phi(A)$. Thus, $A=\phi(A)\subset B$, and  \eqref{lemmaczhGrMqsnhc} applies. This proves \eqref{lemmaczhGrMqsnhd}, and the proof of Lemma~\ref{lemmaczhGrMqsnh} is complete.
\end{proof}

\begin{lemma}\label{lemmaMsrDkTplnnK} 
Let $\alg P=\poset P$ be a finite chain with at least two elements,
and 
let $A,B\in \CoalP$. Let the smallest element and the unique atom of $\alg P$ be denoted by $0$ and $w$, respectively. Let $\alg{P'}$ be the subchain $P'=P\setminus\set 0=\filter w$ with the inherited ordering. 
Then  $B$ covers $A$ (in notation, $A\prec B$) in the coalition lattice $\CoalP$ if and only if one of the following three possibilities hold.
\begin{enumerate}[\qquad\textup{({cov$^\ast$-}}$1$)]
\item\label{paa} $0\notin A$ and $B=A\cup\set 0$.
\item\label{pab} 
\begin{enumerate}[\kern-9pt\upshape(i)]
\item\label{pab1}  Either $0\notin A\cup B$, $B\neq A\cup\set w$, and $A\prec B$ in $\pCoalP$, 
\item\label{pab2} or $0\in A\cap B$, $B\neq A\cup\set w$, and $A\setminus\set 0\prec B\setminus\set 0$ in $\pCoalP$.  
\end{enumerate}
\item\label{pta} $0\in A$, $w\notin A$, and $B=(A\setminus\set 0)\cup\set w$.
\end{enumerate}
\end{lemma}

\begin{proof} It is trivial to see that 
\begin{equation}\left.
\parbox{6cm}{each of  \pref{paa},  \pref{pab}, and  \pref{pta} implies that $A<B$ in $\CoalP$.}\right\}
\label{pbxZMlflKnsgl}
\end{equation}

Next, we claim that 
\begin{equation}\left.
\parbox{9 cm}{for every $X\in \CoalP$, the only extensive $X\to X$ map is the identity map $\idmap X\colon X\to X$, defined by $x\mapsto x$.}
\,\,\right\}
\label{pbxbbhRmfPBm}
\end{equation}
We show this by induction on $|X|$. For $|X|\leq 1$, \eqref{pbxbbhRmfPBm} is clear.  For $|X|>1$ and an arbitrary extensive map $\phi\colon X\to X$, the $\phi$-image of the largest element $b$ of $X$ is necessarily $b$, because $b\leq \phi(b)$. Let $Y=:X\setminus\set b$. The restriction 
$\restrict\phi Y$ of $\phi$ to $Y$ is a $Y\to Y$ map by injectivity, and so $\restrict \phi Y\colon Y\to Y$ is an extensive map. Since $\restrict \phi Y=\idmap Y$ by the induction hypothesis, we obtain that $\phi=\idmap X$, proving \eqref{pbxbbhRmfPBm}.

By Lemma~\ref{lemmaczhGrMqsnh} and \eqref{pbxZMlflKnsgl}, \pref{paa} implies that $A\prec B$. 
Assume \pref{pab1}. Then, since $A\prec B$ in $\pCoalP$, 
Lemma~\ref{lemmaczhGrMqsnh} yields that 
\ref{lemmaczhGrMqsnh}\eqref{lemmaczhGrMqsnha} and \ref{lemmaczhGrMqsnh}\eqref{lemmaczhGrMqsnhb} hold for $A$ and $B$ over $\alg P'$. Since $B\neq A\cup\set w$ and $w$ is the smallest element of $\alg P'$,  \ref{lemmaczhGrMqsnh}\eqref{lemmaczhGrMqsnhc} excludes that $A\subset B$. 
Hence,  \ref{lemmaczhGrMqsnh}\eqref{lemmaczhGrMqsnhb} leads to $|A|=|B|=1+|A\cap B|$; this holds not only over $\alg P'$ but also over $\alg P$. That is,  \ref{lemmaczhGrMqsnh}\eqref{lemmaczhGrMqsnhb} holds for $A$ and $B$ over $\alg P$. 
For $x\in P'$, $\height_{\alg P}(x)=1+\height_{\alg P'}(x)$.
Hence,  for every $X\in\pCoalP$, 
\begin{equation}\left.
\begin{aligned}
\pstr{\alg P}X=|X|+\sum_{u\in X}\height_{\alg P}(u)=
|X|+\sum_{u\in X}(\height_{\alg P'}(u)+1) \cr
= |X|+|X|+\sum_{u\in X}\height_{\alg P'}(u)=|X|+\pstr{\alg P'}X
\end{aligned}\,\,\right\}
\label{alignZhfBRdjzG}
\end{equation}
Thus, using \eqref{alignZhfBRdjzG}, $|A|=|B|$, and that  \ref{lemmaczhGrMqsnh}\eqref{lemmaczhGrMqsnha} holds for $A$ and $B$ in over $\alg P'$, 
\[
\pstr{\alg P}B=|B|+\pstr{\alg P'}B=|A|+\pstr{\alg P'}A +1=\pstr{\alg P}A +1,
\]
that is,  \ref{lemmaczhGrMqsnh}\eqref{lemmaczhGrMqsnha} holds for $A$ and $B$ in over $\alg P$. Hence, \eqref{pbxZMlflKnsgl} and Lemma~\ref{lemmaczhGrMqsnh} imply that $A\prec B$ in $\CoalP$, as required.

Next, when assuming  \pref{pab2}, we are going to reduce the task to  \pref{pab1}, which has just been settled. Namely, observe that  \pref{pab1} holds for $A':=A\setminus\set0$ and $B':=B\setminus\set0$. Apart from slight notational changes,
we derived from this situation that  \ref{lemmaczhGrMqsnh}\eqref{lemmaczhGrMqsnhb} holds for $A'$ and $B'$ with $|A'|=|B'|=1+|A'\cap B'|$ and that \ref{lemmaczhGrMqsnh}\eqref{lemmaczhGrMqsnha} also holds for $A'$ and $B'$ over $\alg P$ (and so $A'\prec B'$ in $\CoalP$ but this is not relevant at this moment). These two facts imply that  \ref{lemmaczhGrMqsnh}\eqref{lemmaczhGrMqsnhb} and  \ref{lemmaczhGrMqsnh}\eqref{lemmaczhGrMqsnha} holds for $A=A'\cup\set 0$ and $B=B'\cup\set0$ over $\alg P$, whereby $A\prec B$ in $\CoalP$ by \eqref{pbxZMlflKnsgl} and Lemma~ \ref{lemmaczhGrMqsnh}, as required.

Finally, if \pref{pta}, then $|A|=|B|=|A\cap B|+1$ and $\str B=\str A +\height(w)-\height(0)=\str A +1$, and so \eqref{pbxZMlflKnsgl} together with
Lemma~ \ref{lemmaczhGrMqsnh} yield the required $A\prec B$ in $\CoalP$. 
We have seen that the disjunction of  \pref{paa},  \pref{pab}, and  \pref{pta} is a sufficient condition of $A\prec B$.

In order the see that the above-mentioned disjunction is a necessary condition, the rest of the proof assumes that $A\prec B$ in $\CoalP$. Note in advance that then
\begin{equation}
\text{our assumption, $A\prec B$, excludes that $B=A\cup\set w$,}
\label{eqtxtMPtLnkKmsshVrb}
\end{equation}
since otherwise  $B\neq A$ would give that $w\notin A$  and so $\str B=\str A +1+\height(w)=\str A+2$, which would contradict Lemma~\ref{lemmaczhGrMqsnh}\eqref{lemmaczhGrMqsnha}.  

By Lemma~\ref{lemmaczhGrMqsnh}, \ref{lemmaczhGrMqsnh}\eqref{lemmaczhGrMqsnha} and \ref{lemmaczhGrMqsnh}\eqref{lemmaczhGrMqsnhb} hold for $A$ and $B$ over $\alg P$. 
According to the containment of $0$ in $A$ and $B$, there are four cases to consider. 
First, assume that $0\notin A$ and $0\notin B$.
Then Lemma~\ref{lemmaczhGrMqsnh}\eqref{lemmaczhGrMqsnhd} 
excludes that $|A|<|B|$, whence  \ref{lemmaczhGrMqsnh}\eqref{lemmaczhGrMqsnhb}  imply that $|A|=|B|=1+|A\cap B|$, which holds also over $\alg P'$. In particular, \ref{lemmaczhGrMqsnh}\eqref{lemmaczhGrMqsnhb} holds over $\alg P'$.
Using  \ref{lemmaczhGrMqsnh}\eqref{lemmaczhGrMqsnha} over $\alg P$ and $|A|=|B|$, and computing by \eqref{alignZhfBRdjzG}, we obtain the validity of \ref{lemmaczhGrMqsnh}\eqref{lemmaczhGrMqsnha} over $\alg P'$. Hence, Lemma~\ref{lemmaczhGrMqsnh} gives that $A\prec B$ in $\pCoalP$. Thus, taking \eqref{eqtxtMPtLnkKmsshVrb} also into account, we obtain that  \pref{pab1} holds.

Second, assume that $0\in A$ and $0\in B$, 
and let $A':=A\setminus\set 0$ and $B':=B\setminus\set 0$. 
Lemma~\ref{lemmaskzsRsznmCsnsT} gives easily that 
$A'< B'$ in $\CoalP$.
From  $A\neq B$ and  Lemma~\ref{lemmaczhGrMqsnh}\eqref{lemmaczhGrMqsnhc}, we conclude that $A\not\subset B$. Thus, Lemma~\ref{lemmaczhGrMqsnh}\eqref{lemmaczhGrMqsnhb} gives that 
$|A|=|B|=1+|A\cap B|$. This implies that $|A'|=|B'|=1+|A'\cap B'|$, that is,  \ref{lemmaczhGrMqsnh}\eqref{lemmaczhGrMqsnhb} holds for $A'$ and $B'$. Since $\str B=\str A+1$ by  \ref{lemmaczhGrMqsnh}\eqref{lemmaczhGrMqsnha}, we obtain that   \ref{lemmaczhGrMqsnh}\eqref{lemmaczhGrMqsnha} holds also for $A'$ and $B'$. Combining these facts with Lemma~\ref{lemmaczhGrMqsnh}, we obtain that  
$A'\prec B'$ in $\CoalP$. The previous paragraph has shown that this yields the validity of \pref{pab1} for $A'$ and $B'$. This fact implies trivially that \pref{pab2} holds for $A$ and~$B$.

Third, assume that $0\in A$ but $0\notin B$. Then  
$A\not\subset B$, whereby Lemma~ \ref{lemmaczhGrMqsnh}\eqref{lemmaczhGrMqsnhb} leads to $|A|=|B|=|A\cap B|+1$. That is, 
$A=X\cup\set a$ and $B=X\cup\set b$ with $a\neq b$, $\set{a,b}\cap X=\emptyset$, $a\notin B$, and $b\notin A$. Since both $0$ and $a$ are in the singleton set $A\setminus X$, we have that  $a=0$, and so Lemma~\ref{lemmaczhGrMqsnh}\eqref{lemmaczhGrMqsnha} yields that 
\begin{align*}
\str X+1+\height(b)&= \str B=\str A+1 \cr
&=\str X+1+\height(0)+1=\str X+2.
\end{align*}
Hence, $\height(b)=1$, that is, $b=w$. Consequently, $B=(A\setminus\set 0)\cup\set w$, $w=b\notin A$,  and  \pref{pta} holds.

Fourth, assume that  $0\notin A$ but $0\in B$. If we had that $|A|=|B|=|A\cap B|+1$, then we would have that 
$A=X\cup\set a$ and $B=X\cup\set b$ with $a\neq b$ and $\set{a,b}\cap X=\emptyset$, whereby $b=0$ and $a\neq 0$ would give that $\str A=\str X + 1 + \height(a)> \str X + 1 + \height(b)=\str B$, contradicting Lemma~\ref{lemmaczhGrMqsnh}\eqref{lemmaczhGrMqsnha}. 
Hence, $A\subset B$ by Lemma~\ref{lemmaczhGrMqsnh}\eqref{lemmaczhGrMqsnhb},  $B=A\cup \set0$ by  Lemma~\ref{lemmaczhGrMqsnh}\eqref{lemmaczhGrMqsnhc},  and \pref{paa} holds.

We have seen that whenever $A\prec B$ in $\CoalP$, then the disjunction of  \pref{paa},  \pref{pab}, and  \pref{pta} holds. This completes the proof of Lemma~\ref{lemmaMsrDkTplnnK}.
\end{proof}

As a preparation, let us recall the following  useful  result of Gr\"atzer~\cite{gG_technical_lemma}.

\begin{lemma}[Gr\"atzer~\cite{gG_technical_lemma}]\label{lemmaGGtechn}
Let $\Theta$ be an equivalence relation on a finite lattice $L$ such that the $\Theta$-blocks are intervals. Then $\Theta$ is a congruence if and only if 
\begin{equation*}
\text{$\forall x,y,z\in L$, if $z\prec x$, $z\prec y$, and  $(z,x)\in\Theta$, then $(y,x\vee y)\in\Theta$,}
\end{equation*}
and dually.
\end{lemma}

\begin{proof}[Proof of Lemma~\ref{lemmazBfmWvVjGtRx}]
By a \emph{doubling congruence} we mean a transitive doubling tolerance, that is, a doubling tolerance that happens to be a congruence. Note that 
\begin{equation}\left.
\parbox{7.2cm}{a doubling tolerance is a doubling congruence if and only if its blocks are pairwise disjoint.}
\right\}
\label{pbxhmzfktLmkKr}
\end{equation}
We prove the lemma by induction on the size $|P|$ of the chain $\alg P$. The base of the induction, $|P|=2$, is trivial, whereby the rest of the proof is devoted to the induction step.
With the notation $L:=\pCoalP$ and $M:=\CoalP$, we know from the induction hypothesis that $L$ is distributive. We 
 need to show that $T$ is a doubling congruence on $L$ and $\dbl LT\cong M$; then Theorem~\ref{thmdoubling} will immediately imply that $M$ is also distributive.

From  \eqref{pbxhmzfktLmkKr} and the definition of $T$ in \eqref{eqTdFlsR}, it follows that 
\begin{equation}\left.
\parbox{6cm}{$T$ is an equivalence relation and each of its blocks is a two-element interval.}\,\,
\right\}
\label{pbxRvncbPhnTThhYs}
\end{equation} 
Furthermore, the distributivity (in fact, the modularity) of $L$ implies that whenever $z\prec x$ and $z\prec y$ in $L$, then $(z,x,y,x\vee y)$ is a covering square in $L$, and dually. Therefore, by Lemma~\ref{lemmaGGtechn}, in order to conclude that $T$ is a doubling congruence, it suffices to show that 
\begin{equation}
\parbox{7.5cm}
{if $T$ collapses an edge of a covering square, then it collapses the opposite edge of the square.}
\label{pbxpMhRmzGsB}
\end{equation}
Let $(A=B\wedge C,B,C,D=B\vee C)$ be a covering square in $L=\pCoalP$. First, assume that a lower edge, say, $A\prec B$, is collapsed by $T$. This means that $B=A\cup\set w$ and $w\notin A$.  Since $A\subset B$ and Lemma~\ref{lemmaczhGrMqsnh}\eqref{lemmaczhGrMqsnhc} allows only one $X$ such that $A\prec X$ and $A\subset X$, it follows from Lemma~\ref{lemmaczhGrMqsnh}\eqref{lemmaczhGrMqsnhb} that, in addition to $B=A\cup\set w$, we have that $|A|=|C|$; note that $w$, the smallest element of $\alg P'$, plays the role of $0$ in Lemma~\ref{lemmaczhGrMqsnh}. Since $w$ is already in $B$, \eqref{lemmaczhGrMqsnhb} and \eqref{lemmaczhGrMqsnhc} of Lemma~\ref{lemmaczhGrMqsnh} give that 
$|B|=|D|$. But then $|C|=|A|<|B|=|D|$,  $C\prec D$, and Lemma~\ref{lemmaczhGrMqsnh}\eqref{lemmaczhGrMqsnhd} yield that $D=C\cup\set w$, whereby $(C,D)\in T$. So, $T$ ``spreads'' from a lower edge to the opposite upper edge.  
Second, assume that an upper edge, say, $C\prec D$, is collapsed by $T$, that is, $D=C\cup\set w$ and $w\notin C$; the argument is almost the same as above. Namely, if we had $|B|<|D|$, then Lemma~\ref{lemmaczhGrMqsnh}\eqref{lemmaczhGrMqsnhd} would give that $D=B\cup \set w$ and we would obtain that $B=D\setminus\set w=C$, a contradiction. Hence, $|B|=|D|$. Since $|A|<|C|$ would contradict $w\notin C$ by Lemma~\ref{lemmaczhGrMqsnh}\eqref{lemmaczhGrMqsnhd},  $|A|=|C|$. Hence,  $|A|=|C|<|D|=|B|$, which together with $A\prec B$ and Lemma~\ref{lemmaczhGrMqsnh}\eqref{lemmaczhGrMqsnhd} yield that $B=A\cup\set w$, whereby $A\prec B$ is collapsed by $T$, as required.  We have seen the validity of \eqref{pbxpMhRmzGsB}, whereby we have shown that $T$ is a doubling congruence on $L$.

Next, we define the following map
\[
\gamma\colon \dbl LT\to M\,\,\text{ by }\,\,(A,k)\mapsto
\begin{cases}
A,&\text{ if }k=0,\cr
A\cup\set 0,&\text{ if }k=1,
\end{cases}
\]
and we are going to show that $\gamma$ is a lattice isomorphism. Since $\gamma$ is trivially a bijection,
it suffices to show that $\gamma$ is an order-isomorphism. Furthermore, since orderings on finite posets are determined by the corresponding covering relations,
our task reduces to proving that $x\prec y$ in $\dbl LT$ if and only if $\gamma(x)\prec \gamma(y)$ in $M$. 
The covering pairs in $\dbl LT$ and those in $M$ are described by \cref{caa}--\cref{cta} from the proof of Theorem~\ref{thmdoubling} and by \pref{caa}--\pref{cta} from Lemma~\ref{lemmaMsrDkTplnnK}, respectively. Therefore,
it suffices to prove that for any $(A,i)$ and $(B,j)$ in $\dbl LT$ and for any $\ell\in\set{1,2,3}$,  
\begin{equation}\left.
\parbox{7cm}{\fcref\ell{} holds for $(A,i)$ and $(B,j)$ if and only if  \fpref\ell{} holds for $\gamma(A,i)$ and $\gamma(B,j)$.}\,\,\right\}
\label{pbxtlRkrdBnhsVk}
\end{equation}

Assume that \cref{caa} holds for $(A,i)$ and $(B,j)$ in 
$\dbl LT$. That is, $(A,i)=(A,0)$ and $(B,j)=(A,1)$. Hence.
$\gamma(A,i)=A$ and $\gamma(B,j)=A\cup\set0$, whereby $\gamma(A,i)$ and $\gamma(B,j)$ satisfy \pref{caa}. Conversely, assume that  $\gamma(A,i)$ and $\gamma(B,j)$ satisfy \pref{caa}. Then $0\notin \gamma(A,i)$ and $0\in \gamma(B,j)=\gamma(A,i)\cup\set 0$ imply that $i=0$ and $j=1$. Hence, by  the definition of $\gamma$, we have that 
$A\cup\set0=\gamma(A,0)\cup \set 0=\gamma(B,1)=B\cup\set 0$, whereby $A=B$. Thus, $(A,i)=(A,0)$ and $(B,j)=(A,1)$ satisfy
 \cref{caa}, as required; this settles \eqref{pbxtlRkrdBnhsVk} for $\ell=1$.

Next, assume that  \cref{cab} holds for $(A,i)$ and $(B,j)$ in $\dbl LT$. That is, $A\prec B$ in $L$ and $i=j$, but $A \npT B$. 
Since  $A \npT B$, we have that $B\neq A\cup\set w$. For the case $i=j=1$, note that $B\neq A\cup\set w$ implies that $B\cup\set 0\neq (A\cup\set 0)\cup\set w$. Thus, no matter if $i=j$ is $0$ or $1$, it follows that  \pref{cab} holds for  $\gamma(A,i)$ and $\gamma(B,j)$. 

Conversely, assume that  \pref{cab} holds for  $\gamma(A,i)$ and $\gamma(B,j)$; there are two cases depending on the containment of $0$ in $\gamma(A,i)$. First, assume that $0\notin\gamma(A,i)$. Then  \pref{pab2} is excluded, so  \pref{pab1} holds for 
 $\gamma(A,i)$ and $\gamma(B,j)$. Hence $0\notin\gamma(B,j)$, 
$i=j=0$, $A=\gamma(A,0)\prec \gamma(B,0)= B$ in $L$,  and $B\neq A\cup\set w$. Let us summarize for later reference that
\begin{equation}
i=j,\,\text{ }\, 
A\prec B\,\text{ in }\, L, \,\text{ and }\,B\neq A\cup\set w. 
\label{eqtxtCttShRk}
\end{equation}
We are going to show that \eqref{eqtxtCttShRk} implies that 
 \cref{cab} holds for $(A,i)$ and $(B,j)$.
Since $A\prec B$ excludes that 
$A=B$, the equality $A\cup\set w=B\cup \set w$ 
would only be possible if we had that $B=A\cup \set w$, excluded above, or $A=B\cup \set w\supset B$, excluded by $A<B$. Thus $A\cup\set w\neq B\cup \set w$, that is, $(A,B)\notin T$ and  $A\npT B$. Therefore, \cref{cab} holds for $(A,i)$ and $(B,j)$, as required. 

Second, assume that $0\in\gamma(A,i)$. Now  \pref{pab1} is excluded, so  \pref{pab2} holds for  $\gamma(A,i)$ and $\gamma(B,j)$. In particular, $0\in\gamma(B,j)$. Hence,
$i=j=1$, $\gamma(A,i)=A\cup\set 0$, and $\gamma(B,j)=B\cup\set 0$. Thus, the validity of \pref{pab2} for these two sets gives that $A=(A\cup\set 0)\setminus\set 0\prec (B\cup\set 0)\setminus\set 0=B$ in $L$ and $B\cup\set 0\neq B\cup\set 0\cup \set w$. Since this non-equality gives that $B\neq A\cup \set w$, \eqref{eqtxtCttShRk} is fulfilled. We already know that 
\eqref{eqtxtCttShRk} implies that  \cref{cab} holds for $(A,i)$ and $(B,j)$. Therefore, we have shown  \eqref{pbxtlRkrdBnhsVk} for $\ell=2$.

Finally,  assume that  \cref{cta} holds for $(A,i)$ and
 $(B,j)$ in $\dbl LT$. That is, $i=1$, $j=0$, and, furthermore,  $A\pT B$, which gives that $w\notin A$ and $B=A\cup\set w$. Hence, $\gamma(A,i)=A\cup\set 0$ and $\gamma(B,j)=B$, and they clearly satisfy  \pref{pta} since 
$(\gamma(A,i)\setminus\set 0)\cup \set w=((A\cup\set 0)\setminus\set 0)\cup \set w=A\cup\set w=B=\gamma(B,j)$. Conversely, assume that \pref{pta} holds for $\gamma(A,i)$ and $\gamma(B,j)$. Then $0\in \gamma(A,i)$ and $0\notin\gamma(B,j)$ yield that 
$i=1$ and $j=0$. Furthermore,  $w\notin \gamma(A,i)=\gamma(A,1)=A\cup\set 0$ give that $w\notin A$, and we also have that
\[B=\gamma(B,j)=(\gamma(A,1)\setminus\set 0)\cup\set w = ((A\cup\set 0)\setminus\set 0)\cup\set w = A\cup\set w .
\]
Hence,  $B\cup \set w=A\cup\set w$ and $(A,B)\in T$. 
Combining this with $A\subset B$ and \eqref{pbxRvncbPhnTThhYs}, we obtain that $A\pT B$. Thus,  \cref{cta} holds for $(A,i)=(A,1)$ and $(B,j)=(B,0)$ in $\dbl LT$. This completes the proof of \eqref{pbxtlRkrdBnhsVk} and that of Lemma~\ref{lemmazBfmWvVjGtRx}.
\end{proof}

\end{document}